\documentclass[graybox]{svmult}

\usepackage{mathptmx}
\usepackage{helvet}
\usepackage{courier}
\usepackage{type1cm}

\usepackage{makeidx}
\usepackage{graphicx}
\usepackage{multicol}
\usepackage[bottom]{footmisc}

\usepackage{bm}

\usepackage{amssymb}
\usepackage{amsmath}
\usepackage{psfrag}
\usepackage{epsf}
\usepackage{mathrsfs}
\usepackage{epstopdf}

\newcommand{\RR}{\mathbb R}
\newcommand{\QQ}{\mathbb Q}
\newcommand{\PP}{\mathbb P}
\newcommand{\ZZ}{\mathbb Z}

\newcommand{\f}[1]{\mathbf{#1}}

\newcommand{\rst}[1]{\ensuremath{ \left.\rule{0ex}{1.5ex}\right|_{#1}}}

\makeindex
\begin{document}

\title*{Construction of Smooth Isogeometric Function Spaces on Singularly Parameterized Domains}
\titlerunning{Smooth Isogeometric Function Spaces on Singularly Parameterized Domains}
\author{Thomas Takacs}
\institute{
Thomas Takacs \at \email{thomas.takacs@oeaw.ac.at} \\ Johann Radon Institute for Computational and Applied Mathematics (RICAM), Austrian Academy of Sciences, Linz, Austria \&  Department of Mathematics, University of Pavia, Italy\\
This is a pre-print of the following work: J.-D. Boissonnat et al. (EDS.): Curves and Surfaces 2014, LNCS 9213, pp. 433-451, 2015, Springer. Reproduced with permission of Springer International Publishing Switzerland. The final authenticated version is available online at: \\http://dx.doi.org/10.1007/978-3-319-22804-4\_30}
\maketitle

\abstract{We aim at constructing a smooth basis for isogeometric function spaces on domains of reduced geometric 
regularity. In this context an isogeometric function is the composition of a piecewise rational function with the 
inverse of a piecewise rational geometry parameterization. We consider two types of singular parameterizations, domains 
where a part of the boundary is mapped onto one point and domains where parameter lines are mapped collinearly at the boundary. 
\newline\indent
We locally map a singular tensor-product patch of arbitrary degree onto a triangular patch, thus splitting the parameterization 
into a singular bilinear mapping and a regular mapping on a triangular domain. This construction yields an isogeometric function 
space of prescribed smoothness. Generalizations to higher dimensions are also possible and are 
briefly discussed in the final section.}

\section{Introduction}
\label{sec:1}

In this paper we are dealing with isogeometric function spaces derived from singular NURBS parameterizations. 
We consider two different configurations of singular planar NURBS geometry parameterizations, leading to two different types of triangular domains. 
The goal of our construction is the definition of arbitrarily smooth isogeometric function spaces defined 
on these domains. The approach presented here can be generalized to other types of domains and to higher dimensions.  

The ability to construct test/ trial functions of high smoothness, suitable for numerical simulations, is one of the main features of isogeometric analysis, as introduced in \cite{Hughes2005}. B-spline and NURBS function spaces on standard tensor-product domains possess the possibility of $k$-refinement, creating a sequence of non-nested spaces of increasing degree and increasing smoothness. Hence, increasing degree and smoothness may lead to improved convergence \cite{BeiraodaVeiga2011}. Several applications in isogeometric analysis rely on function spaces of smoothness of higher order, like differential equations of higher order \cite{Cottrell2006}, or the analysis of shells \cite{Benson2010,Kiendl2009,Kiendl2010}, just to name a few examples. In all these applications, the results may be deteriorated if singular parameterizations are present. To overcome this deficiency, we present constructions leading to isogeometric functions spaces of arbitrary smoothness on singularly parameterized domains. 

We start with some preliminary definitions and notation on B-splines and NURBS in Section \ref{subsec:1} and on isogeometric functions in Section \ref{subsec:2}. The smoothness conditions of interest are presented in Section \ref{subsec:3}. In Section \ref{sectionTypeA} we develop the construction of smooth spaces over singular domains where one edge of the parameter domain is mapped onto one point in the physical domain. In Section \ref{sectionTypeB} we present a similar costruction for domains where two parameter directions are collinear at the boundary of the physical domain. Both constructions can be used to obtain circular domains, see also \cite{Lu2009}. We briefly discuss generalizations to higher dimension in Section \ref{sectionHigherDim} and conclude the presented results in Section \ref{sectionConclusion}.

\section{Preliminaries}
\label{sec:2}

Isogeometric function spaces $\mathcal{V}$, as they are present in isogeometric analysis, are built from an underlying B-spline or NURBS space. Hence, to introduce the notation needed, we start this preliminary section with recalling the notion of B-splines and NURBS. We do not give detailed definitions here and refer to standard literature for further reading \cite{PieglTiller1995,Farin1999,Prautzsch2002}.

\subsection{B-Splines and NURBS}
\label{subsec:1}

Univariate B-splines are piecewise polynomial functions. Given a degree $p\in\ZZ^+$ and a knot vector 
$\mathrm{S}=\left(s_{-p}, \ldots, s_{N+p+1}\right)$ of length $N+2p+2$, the $i$-th B-spline, 
for $i=0,\ldots,N+p$, is denoted by $B^p_{i}[\mathrm{S}](s)$. We assume that the parameter domain 
is the unit interval and that the knot vector is open, i.e. 
\begin{equation}
	0 = s_{-p} = \ldots = s_0 < s_1 \leq \ldots \leq s_N < s_{N+1} = \ldots = s_{N+p+1} = 1.
	\label{equationOpenKV}
\end{equation}
Note that any B-spline basis function can be represented via its local knot vector 
\begin{equation*}
	B^p_i[\mathrm{S}](s) = b[s_{i-p},\ldots,s_i,s_{i+1}](s).
\end{equation*}
Using this notation, the degree $p$ of the B-spline $B^p_i[\mathrm{S}]$ is implicitly 
given by the length $p+2$ of the local knot vector.

The concept of univariate B-splines can easily be generalized to two dimensions via a tensor-product construction. Let $p,q\in\ZZ^+$ and let $\mathrm{S}$ and $\mathrm{T}$ be open knot vectors fulfilling equation (\ref{equationOpenKV}). 
The parameter domain is set to be the box $\f B = [0,1]^2$, leading to 
the \emph{tensor-product B-spline space} 
\begin{equation*}
	\mathcal{S} = \mbox{span} \left( \left\{ B^p_i[\mathrm{S}] \, B^q_j[\mathrm{T}] : \f B \rightarrow \RR 
	\; | \; \mbox{ for } (0,0) \leq (i,j) \leq (N_1+p,N_2+q) \right\} \right).
\end{equation*}
The B-splines span a piecewise polynomial function space on a grid given by the knot vectors $\mathrm{S}$ and $\mathrm{T}$. 
Given a weight function $g_0 \in \mathcal{S}$, with $g_0(\f s) > 0$ for all $\f s \in \f B$, we can define a \emph{NURBS space} via 
\begin{equation*}
	\mathcal{N} = \left\{ \frac{f}{g_0}: \f B \rightarrow \RR \; | \; 
	f \in\mathcal{S} \right\}.
\end{equation*}
We can now define isogeometric function spaces. 

\subsection{Isogeometric functions}
\label{subsec:2}

We use the following standard definition of isogeometric functions over a physical domain $\Omega$, 
where we follow the notation in \cite{Takacs2014}. 
This definition is based on the concept of \emph{isogeometric analysis} introduced in \cite{Hughes2005}. 
For a given NURBS geometry parameterization 
\begin{equation*}
\f G = (G_1,G_2)^T = \left(\frac{g_1}{g_0},\frac{g_2}{g_0}\right)^T: \f B \rightarrow \overline{\Omega} \subset \RR^2,
\end{equation*}
with $G_1,G_2 \in \mathcal{N}$, the space of \emph{isogeometric functions} defined 
on the open domain $\Omega = \f G (\f B^\circ)$ is denoted by 
\begin{equation*}
	\mathcal{V} = \left\{ \varphi: \Omega \rightarrow \RR \; | \; 
	\varphi = F \circ \f G^{-1}, 
	\mbox{ with } F = \frac{f}{g_0} \in\mathcal{N} \right\}.
\end{equation*}
We assume that $\f G$ is invertible in the interior $\f B^\circ = ]0,1[^2$ of the box $\f B$, hence 
the functions $\varphi$ are well-defined. 
Note that an isogeometric function $\varphi$ can be defined via its graph surface in homogeneous coordinates
\begin{equation}
	\f f = (g_{0},g_{1},g_{2},f)^T: \f B \rightarrow \tilde{\Omega} \times \RR,
	\label{equationGraphf}
\end{equation}
with $r_j \in \mathcal{S}$ for $j=0,1,2,3$. Here $\tilde{\Omega}$ is given such that $\Pi(\tilde{\Omega}) = \Omega$, 
where the mapping $\Pi:(x_0,x_1,x_2)\mapsto (x_1/x_0,x_2/x_0)$ is the central projection 
onto the plane $x_0 = 1$.

In the following subsection we present smoothness conditions which are of interest in isogeometric analysis.

\subsection{Smoothness conditions}
\label{subsec:3}

We consider a notion of continuity, which may be of interest for any numerical application where a high order of smoothness is necessary. 
\begin{definition}
The space $\mathscr{C}^{k}(\overline{\Omega})$ of $\mathscr{C}^{k}$-continuous functions on the closure of $\Omega$ is defined as the space 
of functions $\varphi:\Omega\rightarrow\RR$ with $\varphi\in C^k(\Omega)$ such that there exists a unique 
limit 
\begin{equation*}
	\lim_{\f y \rightarrow \f x, \f y \in \Omega}
	\frac{\partial^{|\alpha|} \varphi(\f y)}{\partial x_1^{\alpha_1} \partial x_2^{\alpha_2}} = 
	\frac{\partial^{|\alpha|} \varphi(\f x)}{\partial x_1^{\alpha_1} \partial x_2^{\alpha_2}}
\end{equation*}
for all $\f x \in \partial\Omega = \overline{\Omega} \backslash \Omega$ and for all 
$|\alpha| = \alpha_1 + \alpha_2 \leq k$. Here $C^k(\Omega)$ is the traditional space of $k$-times continuously differentiable functions on the open domain $\Omega$. 
\end{definition}
The highest reasonable smoothness is smoothness of order $k=p-1$, where $p$ is the degree of the 
spline space $\mathcal{S}$. However, this may not be feasible for arbitrary domains. 

Another way to prescribe smoothness is by regularity in the sense of Sobolev spaces, i.e. 
$\mathcal{V}\subset H^{k}(\Omega)$. Note that $\mathscr{C}^{k}(\overline{\Omega}) \subset H^{k}(\Omega)$. However, for many isogeometric function spaces, the condition $\mathcal{V}\subset \mathscr{C}^{k}(\overline{\Omega})$ is equivalent to $\mathcal{V}\subset H^{k+1}(\Omega)$. 
If the domain parameterization $\f G$ is singular somewhere at the boundary $\partial\f B$, then the function space $\mathcal{V}$ may not be regular. For most settings $\mathcal{V} \subset \mathscr{C}^{0}(\overline{\Omega})$ and consequently $\mathcal{V} \subset H^1({\Omega})$ is not fulfilled (e.g. for patches of type A, see Section \ref{sectionTypeA}). For studies concerning Sobolev regularity on singular parameterizations in isogeometric analysis we refer to \cite{Takacs2011,Takacs2012-1}. The papers present construction schemes for $H^1$- and $H^2$-smooth isogeometric function spaces. 

In this paper we generalize the presented approach to $\mathscr{C}^{k}$-smoothness for arbitrary $k$. We consider two types of singular parameterizations. The first type A is a class of singular parameterizations where a part of the boundary of $\f B$ is mapped onto one point. 
The second type B covers singular parameterizations where the parameter lines in the physical domain are collinear at the boundary. 

\section{Singular tensor-product patches of type A}\label{sectionTypeA}

In this section we construct smooth isogeometric function spaces on singular patches of type A. 
A parameterization $\f G$ is called a \emph{singular mapping of type A} if it fulfills 
\begin{equation}
	\det \nabla \f G (\f s) = 0 \;\mbox{ for all }\; \f s \in \{0\}\times [0,1]. \label{equationTypeA}
\end{equation}
Hence, the part of the boundary $\{0\}\times [0,1] \subset \partial \f B$ of the parameter domain box $\f B$ is mapped onto one point in the physical domain. The class of singular patches we consider is derived from triangular B\'ezier patches. We start with a construction for B\'ezier patches which we then generalize to B-spline patches. Note that the condition (\ref{equationTypeA}) is more general then the configurations we consider in this section.

\subsection{Triangular B\'ezier patches as singular tensor-product B\'ezier patches}

As presented by Hu in \cite{Hu2001}, a triangular B\'ezier patch 
\begin{equation*}
	\bm{\rho} \; : \; \Delta \rightarrow \RR^d \; : \;  (u,v) \mapsto \sum_{i+j+k = p} \beta^p_{(i,j,k)}(u,v) \; \bm{\rho}_{i,j,k},
\end{equation*}
with control points $\bm{\rho}_{i,j,k} \in \RR^d$, parameter domain 
\begin{equation*}
	\Delta = \{ (u,v): 0\leq u \leq 1 , \; 0\leq v\leq u \}
\end{equation*}
and basis functions 
\begin{equation*}
	\beta^p_{(i,j,k)} \; : \; \Delta \rightarrow \RR \; : \;  (u,v) \mapsto \frac{p!}{i! j! k!} (1-u)^i v^j (u-v)^k,
\end{equation*}
can be represented as a tensor-product B\'ezier patch 
\begin{equation}
		\f f \; : \; \f [0,1]^2 \rightarrow \RR^d \; : \;  \f s = (s,t) \mapsto \sum^p_{i = 0} \sum^{p}_{j = 0} b^p_{i}(s) \; b^p_{j}(t) \; \f f_{i,j},
		\label{equationTransform}
\end{equation}
with Bernstein polynomials $b^p_{i}$ of degree $p$, where
\begin{equation*}
	\f f_{i,j} = \sum^{i}_{\ell=0}\binom{i}{\ell} \frac{ \binom{p-i}{j-\ell} }{ \binom{p}{j} } \bm{\varrho}_{p-i,\ell,i-\ell}
\end{equation*}
for $0\leq i,j \leq p$. The control points for each row are computed via degree elevation. For $i=0$ all control points are the result of degree elevation of a constant ``curve'', i.e. all points are equal
\begin{equation*}
	\f f_{0,j} = \bm{\rho}_{p,0,0},
\end{equation*}
hence the tensor-product B\'ezier patch $\f f(s,t)$ is singular at $s=0$. For $i=1$ the control points result from degree elevating a linear, for $i=2$ from degree elevating a quadratic curve, and so on.

The transformation leading to Equation (\ref{equationTransform}) can also be interpreted as a change of parameters 
\begin{equation}
	\f f = \bm{\rho} \circ \f u 
	\label{equationComposition}
\end{equation}
with the reparameterization 
\begin{eqnarray*}
(s,t) &\mapsto& (u,v), \;\; \mbox{with} \\ 
\f u(s,t) &=& (u(s,t),v(s,t)) = (s,s\,t)
\end{eqnarray*}
with $\f u \in (\QQ_1(\RR^2))^2$, $\bm{\rho} \in (\PP_k(\RR^2))^d$ and $\f f \in (\QQ_k(\RR^2))^d$. Note that the bilinear mapping $\f u(s,t)$ is singular for $s=0$. Here $\QQ_k(\RR^\ell)$ is the space of $\ell$-variate polynomials of maximal degree $\leq k$ and $\PP_k(\RR^\ell)$ is the space of $\ell$-variate polynomials with total degree $\leq k$. 

Selecting $\bm{\rho}_{p-i,\ell,i-\ell} \in \RR^4$, we get $\f f = \bm{\rho} \circ \f u: \f B\rightarrow \RR^4$ which may serve as a homogeneous graph surface of an isogeometric function as in (\ref{equationGraphf}). In this case we conclude $\mathcal{V} \subset \mathscr{C}^\infty (\overline{\Omega})$ if the rational triangular B\'ezier patch  $(\varrho_1/\varrho_0, \varrho_2/\varrho_0)^T$ is regular. Generalizing this construction we can define a smooth isogeometric function space on a certain class of domains containing a singularity. We will also give a detailed proof of the smoothness result in the following section.

\subsection{Smooth function spaces over a singular B-spline patch}

Given a tensor-product B-spline function space $\mathcal{S}$ of degree $(p,q)$ and prescribed order of smoothness $k \leq \min(p,q)$ we want to construct a function space $\mathcal{S}^k \subset \mathcal{S}$, as well as $\mathcal{V}^k \subset \mathcal{V}$ derived from $\mathcal{S}^k$, such that $\mathcal{V}^k\subset \mathscr{C}^{k}(\overline{\Omega})$. 

In the previous section we constructed a polynomial patch $\f f = \bm{\rho} \circ \f u$ that can be split into a singular bilinear part $\f u$ and a regular polynomial part $\bm{\rho}$ defined on a triangular domain. The core idea of the generalized approach is the following. Given a B-spline surface $\f f \in (\mathcal{S})^4$, we assume that $\f f$ is equivalent to a triangular patch up to order $k$ at the singularity. Hence, we introduce the function space $\mathcal{S}^k(\f u,\f S) \subset \mathcal{S}$. 
\begin{definition}\label{definitionSk}
The function space $\mathcal{S}^k(\f u,\f S) \subset \mathcal{S}$ is defined as the space of splines $f \in \mathcal{S}$, such that there exists a polynomial $\varrho \in \mathbb{P}_k$ fulfilling 
\begin{equation}
	\frac{\partial^{|\alpha|} f }{\partial s^{\alpha_1}\partial t^{\alpha_2}} (\f s) = 
	\frac{\partial^{|\alpha|} (\varrho \circ \f u) }{\partial s^{\alpha_1}\partial t^{\alpha_2}} (\f s) \; \mbox{ for all } \f s \in \f S,
\label{equationftD}
\end{equation}
for $0 \leq \alpha_1, \alpha_2 \leq k$, $|\alpha| = \alpha_1+\alpha_2$, where $\f u$ is the mapping 
\begin{eqnarray*}
	\f u : \;\; [0,1]^2 & \; \rightarrow \; & \Delta = \{ (u,v): 0\leq u \leq 1 , \; 0\leq v\leq u \} \\
	(s,t)^T & \; \mapsto \; & \left(s, s \, t\right)^T
\end{eqnarray*}
and $\f S = \{0\} \times [0,1]$.
\end{definition}
Note that if $k=q=p$, then $\mathcal{S}^p(\f u,\f S) \rst{[0,s_1]\times[0,1]} \circ \f u^{-1} = \PP_{p}$, where $\PP_{p}$ is the space of polynomials of total degree $\leq p$ and $s_1$ is the first interior knot of the knot vector $\mathrm{S}$. 

Using this approach, we obtain linear conditions on the B-spline 
basis functions. Certain linear combinations of B-spline basis functions will correspond to the 
basis functions on the triangular patch. Just as presented in \cite{Hu2001}, Definition \ref{definitionSk} is equivalent to 
the first row of control points being constant, the second row forming a linear curve, the third row a quadratic, and so on. 
This leads to the following definition.
\begin{definition}\label{definitionSkhat}
Let $k\leq \min(p,q)$. The basis $\mathbb{S}^k$ is defined via 
\begin{equation*}
\begin{array}{ll}
	\mathbb{S}^k = & \left\{ B^p_i[\mathrm{S}](s) b^i_{j}(t): 0\leq j\leq i \mbox{ and } 0\leq i\leq k \right\} \\ 
	& \cup \left\{ B^p_i[\mathrm{S}](s)B^q_j[\mathrm{T}](t): k+1\leq i\leq N_1+p \mbox{ and } 0\leq j\leq N_2+q \right\},
\end{array}
\end{equation*}
where $B^p_i[\mathrm{S}]B^q_j[\mathrm{T}]$, with $(0,0) \leq (i,j) \leq (N_1+p,N_2+q)$, is the standard basis of $\mathcal{S}$ and $b^i_{j}(t)$ is the $j$-th Bernstein polynomial of degree $i$. 
\end{definition}
Figure \ref{figureIndexset} gives a schematic depiction of the index set corresponding to $\mathbb{S}^k$ for $k=3$.
\begin{figure}[!ht]
    \centering
    \psfrag{i}[cc]{\scalebox{1}{$i$}}
    \psfrag{j}[cc]{\scalebox{1}{$j$}}
    \psfrag{k}[cc]{\scalebox{1}{$k$}}
    \psfrag{0}[cc]{\scalebox{1}{$0$}}
    \psfrag{1}[cc]{\scalebox{1}{$1$}}
    \psfrag{2}[cc]{\scalebox{1}{$2$}}
    \psfrag{3}[cc]{\scalebox{1}{$3$}}
    \includegraphics[width=0.4\textwidth]{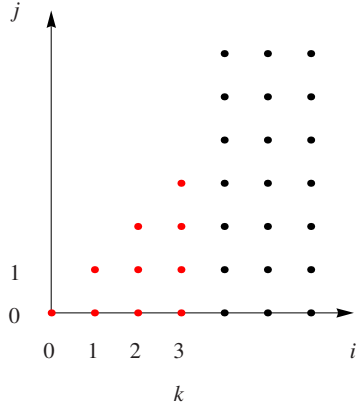}
    \caption{Index set corresponding to $\mathbb{S}^k$ for $k=3$}\label{figureIndexset}
\end{figure}
\begin{lemma}\label{lemma1}
Let $k\leq \min(p,q)$. The set $\mathbb{S}^k$ given in Definition \ref{definitionSkhat} is a basis for the space $\mathcal{S}^k(\f u,\f S)$ given in Definition \ref{definitionSk}. 
\end{lemma}
\begin{proof}
Obviously, we have $\mbox{span}(\mathbb{S}^k) \subset \mathcal{S}$ since
\begin{equation*}
	 b^i_{j}(t) \in \mbox{span} \left( \left\{ B^q_j[\mathrm{T}](t): 0\leq j\leq N_2+q \right\} \right)
\end{equation*}
for all $i,j$ with $0 \leq j\leq i\leq q$.

We first show that $\mbox{span}(\mathbb{S}^k) \subseteq \mathcal{S}^k(\f u, \f S)$, i.e. all functions $f \in \mathbb{S}^k$ fulfill equation (\ref{equationftD}) for some polynomial $\varrho \in \PP_{\min(p,q)}$. Since the condition (\ref{equationftD}) needs to be fulfilled for $s = 0$, we assume that $s<s_1$, which is the first knot of the knot vector $\mathrm{S}$. For $i>k+1$, the functions $B^p_i[\mathrm{S}](s)B^q_j[\mathrm{T}](t)$ fulfill $\frac{\partial^{\alpha}}{\partial s^{\alpha}} B^p_i[\mathrm{S}](s)B^q_j[\mathrm{T}](t) = 0$ for all $\alpha \leq k$. Hence (\ref{equationftD}) is fulfilled with $\varrho \equiv 0$. For $B^p_i[\mathrm{S}](s) b^i_{j}(t)$ with $j \leq i\leq k$ we have that $B^p_i[\mathrm{S}](s) = s^i \, r(s)$, where $r(s)$ is some polynomial in $s$ of degree $\leq k-i$. Moreover, $b^i_{j}(t)$ is a polynomial in $t$ of degree $i$. Hence, $B^p_i[\mathrm{S}](s) b^i_{j}(t) = r(s) \, s^i \, b^i_{j}(t)$ can be represented as a polynomial $\varrho$ in $u=s$ and $v=s\,t$ with total degree $\leq k$. This is exactly the form $B^p_i[\mathrm{S}](s) b^i_{j}(t) = \varrho \circ \f u$ required in equation (\ref{equationftD}). 

What remains to be shown is that $\mathcal{S}^k(\f u, \f S) \subseteq \mbox{span}(\mathbb{S}^k)$. Assume that $f \in \mathcal{S}$ is equivalent to a monomial $u^i\,v^j$, $i+j\leq k$, with respect to condition (\ref{equationftD}). In that case we conclude $f(s,t) = s^{i+j}\, t^j + s^{k+1}\,r(s,t)$ for $s < s_1$, where $r$ is some polynomial of degree $p-k-1$ in $s$ and of degree $q$ in $t$. One can show easily that $f \in \mbox{span}(\mathbb{S}^k)$ in that case. Finally, if $f \in \mathcal{S}$ is equivalent to $\varrho \equiv 0$ with respect to (\ref{equationftD}), then $f(s,t) = s^{k+1}\,r(s,t)$ for $s < s_1$, and again $f \in \mbox{span}(\mathbb{S}^k)$, which concludes the proof. \hfill $\square$
\end{proof}

One can show that if $\bm{\rho}$ is regular and $\f G$ is $C^k$-smooth, then the mapping 
\begin{equation*}
	\f F: \Delta \rightarrow \RR \; \mbox{ with } \; 
	\f F = \f G \circ \f u^{-1}
\end{equation*}
is a $C^k$-smooth mapping from the triangle $\Delta$ to the domain $\overline{\Omega}$. 
Here, the inverse of the singular mapping $\f u$ is equal to 
\begin{equation*}
	\f u^{-1}(u,v) = \left(u,\frac{v}{u}\right)^T.
\end{equation*}
This leads to a split of the mapping $\f G$ into a bilinear singular transformation $\f u$ and a regular mapping $\f F$, via 
\begin{equation*}
	\f G = \f F \circ \f u.
\end{equation*}
The various introduced mappings and domains are depicted in Figure \ref{figureMapping}. 
\begin{figure}[!ht]
    \centering
    \psfrag{s}[cc]{\scalebox{1}{$s$}}
    \psfrag{t}[cc]{\scalebox{1}{$t$}}
    \psfrag{u}[cc]{\scalebox{1}{$u$}}
    \psfrag{v}[cc]{\scalebox{1}{$v$}}
    \psfrag{B}[cc]{\scalebox{1}{$[0,1]^2$}}
    \psfrag{D}[cc]{\scalebox{1}{$\Delta$}}
    \psfrag{O}[cc]{\scalebox{1}{$\overline{\Omega}$}}
    \psfrag{F}[cc]{\scalebox{1}{$\f F$}}
    \psfrag{G}[cc]{\scalebox{1}{$\f G$}}
    \psfrag{U}[cc]{\scalebox{1}{$\f u$}}
    \includegraphics[width=0.6\textwidth]{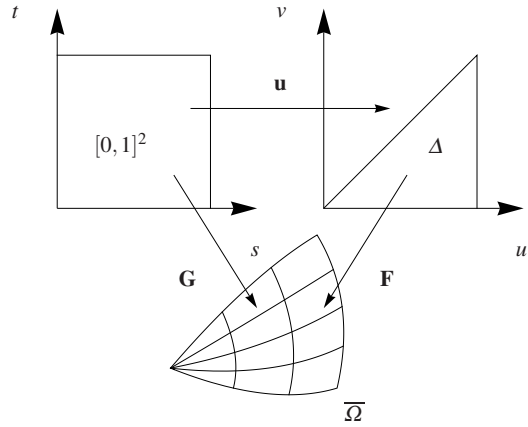}
    \caption{Mappings $\f F$, $\f u$, $\f G$ for an example domain of type A}\label{figureMapping}
\end{figure}

Using this definition we can construct an isogeometric function space fulfilling $\mathcal{V}^k \subseteq \mathcal{V} \cap \mathscr{C}^k(\overline{\Omega})$, for $k\leq \min(p,q)$.
\begin{theorem}
Let $k\leq \min(p,q)$, let $\mathcal{S} \subset C^k(\f B)$ and let $\mathcal{V}^k$ be the isogeometric function space derived from $\mathcal{S}^k(\f u,\f S)$ with $\f G = (g_1/g_0,g_2/g_0)^T$ with $g_0,g_1,g_2\in \mathcal{S}^k(\f u,\f S)$. Moreover, $\f G(s,t)$ is regular for all $s>0$ and $t\in[0,1]$.

Then $\mathcal{V}^k \subset \mathscr{C}^k(\overline{\Omega})$ if $(g_0,g_1,g_2)^T$ is equivalent to $(\varrho_0,\varrho_1,\varrho_2)^T$ with respect to (\ref{equationftD}) and $\left(\frac{\varrho_1}{\varrho_0},\frac{\varrho_2}{\varrho_0}\right)^T$ is regular in $\Delta$.
\end{theorem}
\begin{proof}
Given an isogeometric function $\varphi = f \circ \f G^{-1} \in\mathcal{V}^k$. Due to Lemma \ref{lemma1} the homogeneous graph surface $\f f = (g_0,g_1,g_2,f)^T$ fulfills (\ref{equationftD}) for some $(\varrho_0,\varrho_1,\varrho_2,\varrho_3)^T$. The condition $\varphi\in\mathscr{C}^{k}(\overline{\Omega})$ is
given by 
\begin{equation*}
	\frac{f}{g_0} \circ \left(\frac{g_1}{g_0},\frac{g_2}{g_0}\right)^{-1} \in \mathscr{C}^{k}(\overline{\Omega}). 
\end{equation*}
Since $\f G$ is regular for $s>0$ and $\f G \in C^k(\f B)$, we conclude that $\f G^{-1} \in \mathscr{C}^{k}(\f G([\epsilon,1]\times[0,1]))$ by definition. Hence, it remains to be shown that 
\begin{equation*}
	\frac{f}{g_0} \circ \left(\frac{g_1}{g_0},\frac{g_2}{g_0}\right)^{-1} \in \mathscr{C}^{k}(\f G([0,\epsilon]\times[0,1])). 
\end{equation*}
Due to (\ref{equationftD}) this is equivalent to 
\begin{equation*}
	\frac{\varrho_3}{\varrho_0} \circ \f u^{-1} \circ \f u \circ \left(\frac{\varrho_1}{\varrho_0},\frac{\varrho_2}{\varrho_0}\right)^{-1} \in \mathscr{C}^{k}(\f G([0,\epsilon]\times[0,1])). 
\end{equation*}
Since $\varrho_i \in C^\infty$, this condition is equivalent to $\left(\frac{\varrho_1}{\varrho_0},\frac{\varrho_2}{\varrho_0}\right)$ being invertible 
which concludes the proof.  \hfill $\square$
\end{proof}

In Definition \ref{definitionSkhat} we have already given a basis $\mathbb{S}^k$ for the function space $\mathcal{S}^k(\f u,\f S)$. In the next section we propose an algorithm to determine the coefficients of the linear conditions with respect to the standard basis of $\mathcal{S}$ yielding the new basis functions in $\mathbb{S}^k$. 

\subsection{Algorithm to construct the new basis functions}

In this section we describe an algorithm to find a representation for the new basis functions 
\begin{small}
\begin{equation*}
	\begin{array}{ccccc}
	& & & & B^p_k[\mathrm{S}](s)\, b[0,1,\ldots,1](t) \\	
	& & B^p_2[\mathrm{S}](s)\, b[0,1,1,1](t) & & \vdots \\
	& B^p_1[\mathrm{S}](s)\, b[0,1,1](t) & B^p_2[\mathrm{S}](s)\, b[0,0,1,1](t) & & \\
	B^p_0[\mathrm{S}](s)\, b[0,1](t) & B^p_1[\mathrm{S}](s)\, b[0,0,1](t) & B^p_2[\mathrm{S}](s)\, b[0,0,0,1](t) & \ldots & B^p_k[\mathrm{S}](s)\, b[0,\ldots,0,1](t)
	\end{array}
\end{equation*}
\end{small}
or equivalently 
$$
	\left\{ b[s_{i-p},\ldots,s_{i+1}](s) b^i_{j}(t): 0\leq j\leq i \mbox{ and } 0\leq i\leq k \right\},
$$
in terms of the basis $B^p_{i}[\mathrm{S}](s)B^q_{j}[\mathrm{T}](t)$ of the space $\mathcal{S}$. Here $b^i_{j}(t)$ is the $j$-th 
Bernstein polynomial of degree $i$. 
The algorithm is composed of three steps: degree elevation for $t$, knot insertion for $t$ and tensor-product multiplication with basis functions in $s$-direction. The presented algorithms are taken from standard literature \cite{PieglTiller1995,Prautzsch2002}.

\paragraph*{1. Perform degree elevation in $t$-direction}
Let $\mathrm{E^q_i}$ be the matrix corresponding to the degree elevation of a Bernstein polynomial of degree $i$ represented in terms of Bernstein polynomials of degree $q>i$, i.e.
$$
(b^i_{j}(t))^T_{j=0,\ldots,i} = \mathrm{E^q_i} (b^q_{j}(t))^T_{j=0,\ldots,q}.
$$
The matrix $\mathrm{E^q_i}$ is of dimension $(i+1)\times(q+1)$ and has the form 
$$
\mathrm{E^q_i} = \left( \binom{i}{\ell} \frac{ \binom{q-i}{j-\ell} }{ \binom{q}{j} } \right)_{\ell=0,\ldots,i \; \times \; j=0,\ldots,q},
$$
see e.g. \cite{Prautzsch2002}. In Section \ref{subsectionExampleSingular} we list some of the matrices $\mathrm{E^q_i}$ for example configurations.

\paragraph*{2. Perform knot insertion for $t$-direction}
Let $\mathrm{K_{\tau}}$ be the matrix corresponding to the knot insertion of interior knots $\tau = (t_1,t_2,\ldots t_N)$ of the knot vector $\mathrm{T}$, leading to 
$$
(b^q_{j}(t))^T_{j=0,\ldots,q} = \mathrm{K_{\tau}} (b[t_{j-q},\ldots,t_j,t_{j+1}](t))^T_{j=0,\ldots,N+q}.
$$

The knot insertion matrix $\mathrm{K_{\tau}}$ can be defined via iteratively inserting the knots into the knot vector. The following algorithm gives the resulting matrix for insertion of the single knot $t_i$ into the given knot vector with interior knots up to $t_{i-1}$. A B-spline of degree $q$ with knot vector $(0,\ldots,0,t_1,\ldots,t_{i-1},1,\ldots,1)$ can be represented as a B-spline of degree $q$ with knot vector $(0,\ldots,0,t_1,\ldots,t_{i-1},t_i,1,\ldots,1)$ using knot insertion. The corresponding transformation matrix $\mathrm{K_{t_i}}$ is given by
\begin{small}
\begin{equation*}
	\left( \begin{array}{r} \vdots 
	\\ b[t_{i-q-2},\ldots,t_{i-1}](t) 
	\\ b[t_{i-q-1},\ldots,t_{i-1},1](t) 
	\\ b[t_{i-q},\ldots,1,1](t) 
	\\ \vdots 
	\\ b[t_{i-2},t_{i-1},1,\ldots,1](t) 
	\\ b[t_{i-1},1,\ldots,1](t) 
	\end{array} \right) = 
	\left( \begin{array}{ccccccc}
	\ddots &  &  &  &  &  & \vdots  \\
	 & 1 & 0 & 0 &  & 0 & 0  \\
	 & 0 & 1 & 1-\lambda_{i-q-1} &  & 0 & 0 \\
	 & 0 & 0 & \lambda_{i-q-1} &  & 0 & 0 \\
	 &  &  &   & \ddots &  &  \\
	 & 0 & 0 & 0 &  & 1-\lambda_{i-1} & 0 \\
	\ldots & 0 & 0 & 0 &   & \lambda_{i-1} & 1 
	\end{array} \right)
	\left( \begin{array}{r} \vdots 
	\\ b[t_{i-q-2},\ldots,t_{i-1}](t) 
	\\ b[t_{i-q-1},\ldots,t_{i}](t) 
	\\ b[t_{i-q},\ldots,t_{i},1](t) 
	\\ \vdots 
	\\ b[t_{i-1},t_{i},1,\ldots,1](t) 
	\\ b[t_{i},1,\ldots,1](t) 
	\end{array} \right),
\end{equation*}
\end{small}
where 
\begin{equation*}
	\lambda_{j} = \frac{t_i - t_{j}}{1 - t_{j}}.
\end{equation*}
Using this construction, the knot insertion matrix $\mathrm{K_{\tau}}$ is given via 
$$
\mathrm{K_{\tau}} = \mathrm{K_{t_1}} \mathrm{K_{t_2}} \ldots \mathrm{K_{t_N}}.
$$
In the last step we multiply with the corresponding $s$-dependent functions.

\paragraph*{3. Compute tensor-product basis}
Combining steps 1 and 2 with the tensor product representation leads to 
\begin{eqnarray}
(B^{p}_{i}[\mathrm{S}](s) b^i_{j}(t))^T_{j=0,\ldots,i} &=& \mathrm{E^q_i} (B^{p}_{i}[\mathrm{S}](s) b^q_j(t))^T_{j=0,\ldots,q} \nonumber \\
 &=& \mathrm{E^q_i} \mathrm{K_{\tau}} ( B^{p}_{i}[\mathrm{S}](s) \; B^{q}_{j}[\mathrm{T}](t) )^T_{j=0,\ldots,N+q},\nonumber
\end{eqnarray}
for $0 \leq i \leq p$.

In the following we compute the coefficients for some example configurations.

\subsection{Some example configurations}\label{subsectionExampleSingular}

We start with an example of a patch of degree $p=q=2$. 
\begin{example}
Let $\mathrm{S}=(0,0,0,h,2h,3h,\ldots)$ and $\mathrm{T}=(0,0,0,1/4,1/2,3/4,1,1,1)$. We want to find a representation of the basis $\mathbb{S}^2$ for $\mathcal{S}^2(\f u, \f S)$ with respect to the standard basis $B^2_{i}[\mathrm{S}]B^2_{j}[\mathrm{T}]$ of $\mathcal{S}$. We denote the new basis functions by $\tilde{B}^p_{(i,j)} = B^{p}_{i}[\mathrm{S}](s) b^i_{j}(t)$. 
The basis functions we need to construct are $\tilde{B}^2_{(i,j)}$, with $0\leq i\leq 2$ and $0\leq j\leq i$.

The degree elevation matrices $\mathrm{E^q_i}$ for $q=2$ are given by 
\begin{equation*}
	\mathrm{E^2_0} = \left( \begin{array}{lll} 1 & 1 & 1 \end{array} \right),
	\mathrm{E^2_1} = \left( \begin{array}{lll} 1 & \frac{1}{2} & 0 \\ 0 & \frac{1}{2} & 1 \end{array} \right),
	\mathrm{E^2_2} = \left( \begin{array}{lll} 1 & 0 & 0 \\ 0 & 1 & 0 \\ 0 & 0 & 1 \end{array} \right)
\end{equation*}
and the knot insertion matrix $K_t$, with $(b^2_i)^T_{i=0,1,2} = K_t (B^2_j[\mathrm{T}])^T_{j=0,\ldots,5}$, fulfills 
\begin{equation*}
	K_t = \left( \begin{array}{ccccccc}
	1 & \frac{3}{4} & \frac{3}{8} & \frac{1}{8} & 0 & 0 \\
	0 & \frac{1}{4} & \frac{1}{2} & \frac{1}{2} & \frac{1}{4} & 0 \\
	0 & 0 & \frac{1}{8} & \frac{3}{8} & \frac{3}{4} & 1 
	\end{array} \right).
\end{equation*}
Hence we conclude 
\begin{equation*}
	\tilde{B}^2_{(0,0)}  = \mathrm{E^2_0} K_t 
	(B^2_0[\mathrm{S}] B^2_j[\mathrm{T}])^T_{j=0,\ldots,5} = 
	\left( \begin{array}{ccccccc}
	1 & 1 & 1 & 1 & 1 & 1 
	\end{array} \right)
	(B^2_0[\mathrm{S}] B^2_j[\mathrm{T}])^T_{j=0,\ldots,5},
\end{equation*}

\begin{equation*}
	\left( \begin{array}{r}
	\tilde{B}^2_{(1,0)}  
	\\ \tilde{B}^2_{(1,1)}  
	\end{array} \right)  = \mathrm{E^2_1} K_t 
	(B^2_1[\mathrm{S}] B^2_j[\mathrm{T}])^T_{j=0,\ldots,5} = 
	\left( \begin{array}{ccccccc}
	1 & \frac{7}{8} & \frac{5}{8} & \frac{3}{8} & \frac{1}{8}  & 0 \\
	0 & \frac{1}{8} & \frac{3}{8} & \frac{5}{8} & \frac{7}{8} & 1 
	\end{array} \right)
	(B^2_1[\mathrm{S}] B^2_j[\mathrm{T}])^T_{j=0,\ldots,5}
\end{equation*}
and
\begin{equation*}
	\left( \begin{array}{r}
	\tilde{B}^2_{(2,0)}  
	\\ \tilde{B}^2_{(2,1)}  
	\\ \tilde{B}^2_{(2,2)}  
	\end{array} \right) = \mathrm{E^2_2} K_t
	(B^2_2[\mathrm{S}] B^2_j[\mathrm{T}])^T_{j=0,\ldots,5}
	= 
	\left( \begin{array}{ccccccc}
	1 & \frac{3}{4} & \frac{3}{8} & \frac{1}{8} & 0 & 0 \\
	0 & \frac{1}{4} & \frac{1}{2} & \frac{1}{2} & \frac{1}{4} & 0 \\
	0 & 0 & \frac{1}{8} & \frac{3}{8} & \frac{3}{4} & 1 
	\end{array} \right)
	(B^2_2[\mathrm{S}] B^2_j[\mathrm{T}])^T_{j=0,\ldots,5}.
\end{equation*}
\end{example}
The left hand side of figure \ref{figureSingularPatches} depicts a schematic overview of the bi-quadratic patch. The newly defined basis functions are visualized via their Greville ascissae (red dots). The part of the domain containing the singularity is the red triangle to the left. The part colored in light red is the support of the newly defined basis functions. The remaining part of the patch is not influenced by the modification of the function space. One standard basis function is visualized via its Greville abscissa and support (colored in blue and light blue, respectively).
\begin{figure}[!ht]
    \centering
    \includegraphics[width=0.3\textwidth]{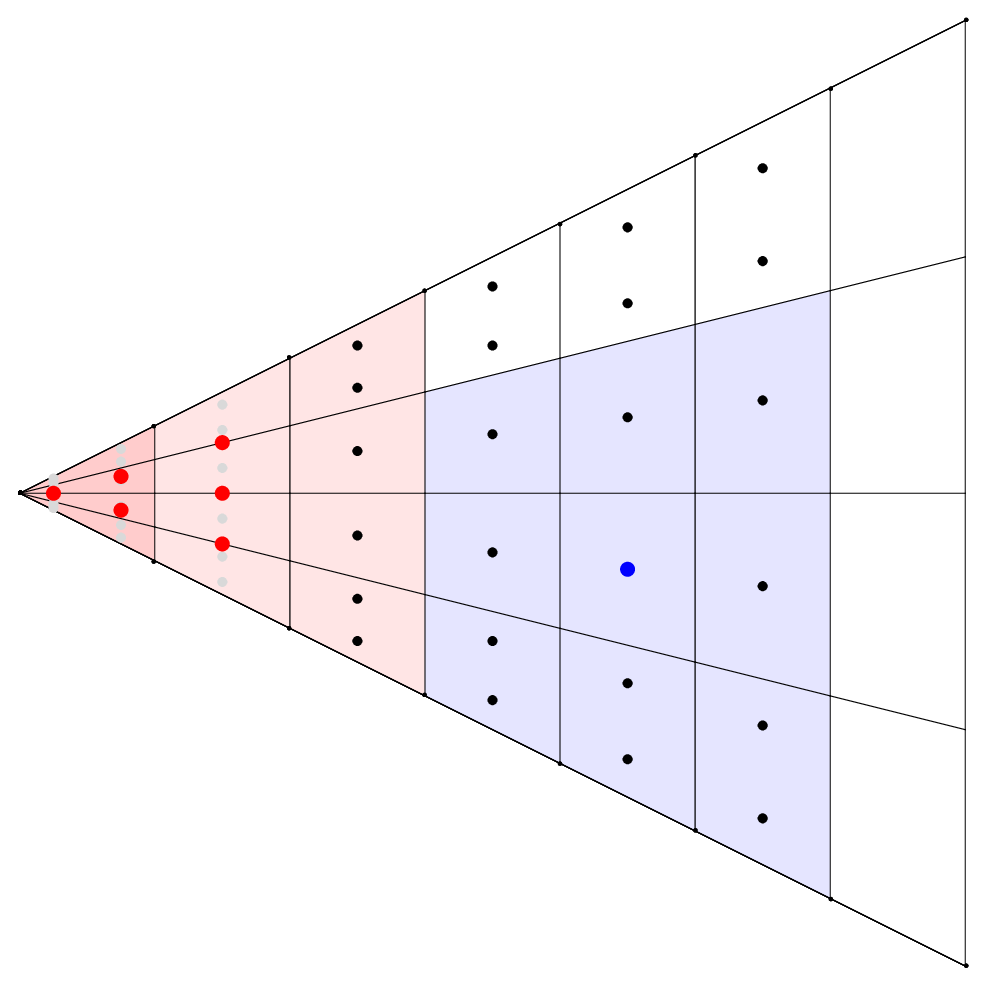}\hspace{0.1\textwidth}
    \includegraphics[width=0.3\textwidth]{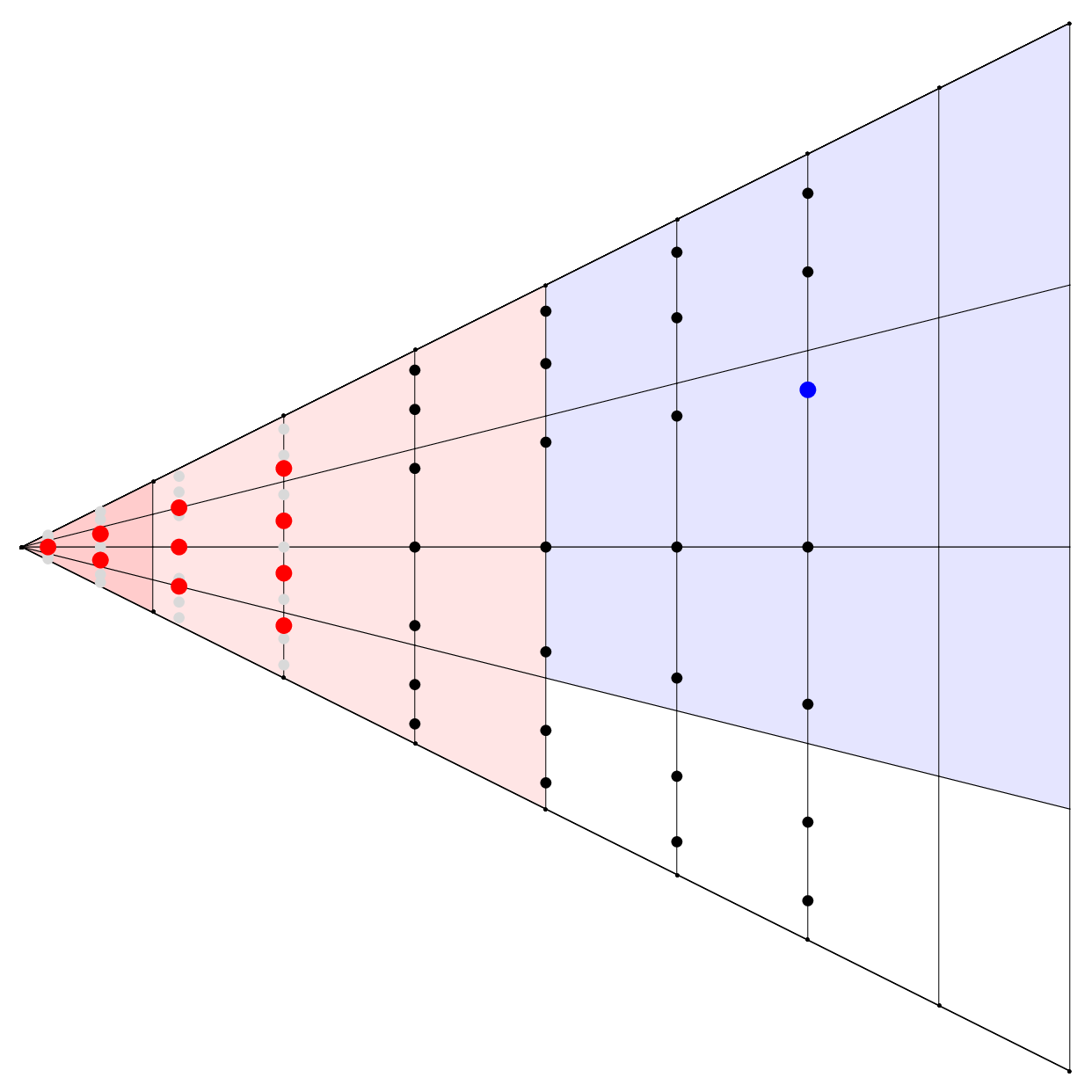}
    \caption{Quadratic (left) and cubic (right) singular B-spline patch of type A}\label{figureSingularPatches}
\end{figure}

The second example is a patch of degree $p=q=3$.
\begin{example}
Let $\mathrm{S}$ and $\mathrm{T}$ be the same knot vectors as in the previous example. Again, we represent the new basis functions $\tilde{B}^3_{(i,j)}$, with $0\leq i\leq 3$ and $0\leq j\leq i$, of $\mathcal{S}^k$ with respect to the standard basis $B^3_{i}[\mathrm{S}]B^3_{j}[\mathrm{T}]$ of $\mathcal{S}$.

Note that in general the degree elevation matrix $\mathrm{E^q_0}$ is a row vector of length $q+1$ with entry $1$ in each column. This derives from the fact that $\mathrm{E^q_0}$ arises from degree elevation of a constant function. Obviously, the matrix $\mathrm{E^q_q}$ is the unit matrix of size $(q+1)\times(q+1)$. 
The remaining degree elevation matrices and the knot insertion matrix $K_t$ fulfill 
\begin{equation*}
\begin{array}{lll}
	\mathrm{E^3_1} = \left( \begin{array}{cccc} 1 & \frac{2}{3} & \frac{1}{3} & 0 \\ 0 & \frac{1}{3} & \frac{2}{3} & 1 \end{array} \right), & 
	\hspace{10pt}
	\mathrm{E^3_2} = \left( \begin{array}{cccc} 1 & \frac{1}{3} & 0 & 0 \\ 0 & \frac{2}{3} & \frac{2}{3} & 0 \\ 0 & 0 & \frac{1}{3} & 1 \end{array} \right), & 
	\hspace{10pt}
	K_t = 
	\left( \begin{array}{cccccccc}
	1 & \frac{3}{4} & \frac{3}{8} & \frac{3}{32} & 0 & 0 & 0 \\
	0 & \frac{1}{4} & \frac{1}{2} & \frac{13}{32} & \frac{1}{8} & 0 & 0 \\
	0 & 0 & \frac{1}{8} & \frac{13}{32} & \frac{1}{2} & \frac{1}{4} & 0 \\
	0 & 0 & 0 & \frac{3}{32} & \frac{3}{8} & \frac{3}{4} & 1 
	\end{array} \right).
\end{array}
\end{equation*}
Hence we conclude 
\begin{equation*}
	\tilde{B}^3_{(0,0)} = 
	\left( \begin{array}{ccccccc}
	1 & 1 & 1 & 1 & 1 & 1 & 1
	\end{array} \right)
	(B^3_0[\mathrm{S}] B^3_j[\mathrm{T}])^T_{j=0,\ldots,6},
\end{equation*}
\begin{equation*}
	\left( \begin{array}{r}
	\tilde{B}^3_{(1,0)}  
	\\ \tilde{B}^3_{(1,1)}  
	\end{array} \right) = 
	\left( \begin{array}{ccccccc}
	1 & \frac{11}{12} & \frac{3}{4} & \frac{1}{2} & \frac{1}{4} & \frac{1}{12}  & 0 \\
	0 & \frac{1}{12} & \frac{1}{4} & \frac{1}{2} & \frac{3}{4} & \frac{11}{12} & 1 
	\end{array} \right)
	(B^3_1[\mathrm{S}] B^3_j[\mathrm{T}])^T_{j=0,\ldots,6}
\end{equation*}
and
\begin{equation*}
	\left( \begin{array}{r}
	\tilde{B}^3_{(2,0)}  
	\\ \tilde{B}^3_{(2,1)}  
	\\ \tilde{B}^3_{(2,2)}  
	\end{array} \right) = 
	\left( \begin{array}{ccccccc}
	1 & \frac{5}{6} & \frac{13}{24} & \frac{11}{48} & \frac{1}{24} & 0 & 0 \\
	0 & \frac{1}{6} & \frac{5}{12} & \frac{13}{24} & \frac{5}{12} & \frac{1}{6} & 0 \\
	0 & 0 & \frac{1}{24} & \frac{11}{48} & \frac{13}{24} & \frac{5}{6} & 1 
	\end{array} \right)
	(B^3_2[\mathrm{S}] B^3_j[\mathrm{T}])^T_{j=0,\ldots,6}
\end{equation*}
as well as
\begin{equation*}
	\left( \begin{array}{r}
	\tilde{B}^3_{(3,0)}  
	\\ \tilde{B}^3_{(3,1)}  
	\\ \tilde{B}^3_{(3,2)}  
	\\ \tilde{B}^3_{(3,3)}  
	\end{array} \right) = 
	\left( \begin{array}{ccccccc}
	1 & \frac{3}{4} & \frac{3}{8} & \frac{3}{32} & 0 & 0 & 0 \\
	0 & \frac{1}{4} & \frac{1}{2} & \frac{13}{32} & \frac{1}{8} & 0 & 0 \\
	0 & 0 & \frac{1}{8} & \frac{13}{32} & \frac{1}{2} & \frac{1}{4} & 0 \\
	0 & 0 & 0 & \frac{3}{32} & \frac{3}{8} & \frac{3}{4} & 1 
	\end{array} \right)
	(B^3_3[\mathrm{S}] B^3_j[\mathrm{T}])^T_{j=0,\ldots,6}
\end{equation*}
\end{example}
The right hand side of Figure \ref{figureSingularPatches} depicts a schematic overview of the bi-cubic patch. The structure is the same as for the previous example.

A simple consequence of the tensor-product structure of the newly defined basis $\mathbb{S}^k$ of the function space $\mathcal{S}^k(\f u,\f S)$ is that we can also define a corresponding dual basis. 

\subsection{Dual basis}

In this section we present a construction of a dual basis for the basis presented in Definition \ref{definitionSkhat}. Recall that the basis $\mathbb{S}^k$ is given by 
\begin{equation*}
\begin{array}{ll}
	 & \left\{ B^p_i[\mathrm{S}](s) b^i_{j}(t): 0\leq j\leq i \mbox{ and } 0\leq i\leq k \right\}  \\ 
	\cup & \left\{ B^p_i[\mathrm{S}](s)B^q_j[\mathrm{T}](t): 0\leq j\leq N+p \mbox{ and } k+1\leq i\leq M+q \right\} .
\end{array}
\end{equation*}
Let $\{\lambda^s_{\ell}\}_{\ell=0,\ldots,M+p}$ be a dual basis of $\{B^p_i[\mathrm{S}](s)\}_{i=0,\ldots,M+p}$ and let $\{\lambda^t_{\ell}\}_{\ell=0,\ldots,N+q}$ be a dual basis of $\{B^q_j[\mathrm{T}](t)\}_{j=0,\ldots,N+q}$, with 
\begin{equation*}
	\lambda^s_{\ell} (B^p_i[\mathrm{S}](s)) = \delta^\ell_i \mbox{ and } \lambda^t_{\ell} (B^q_j[\mathrm{T}](t)) = \delta^\ell_j.
\end{equation*}
One possibility for such a dual basis is presented in \cite{Schumaker2007}. Moreover, let $\{\mu^{i}_{\ell}\}_{\ell=0,\ldots,i}$ be a dual basis to the Bernstein polynomials $\{b^i_{j}(t)\}_{j=0,\ldots,i}$ of degree $i$, with 
\begin{equation*}
	\mu^{i}_{\ell} (b^i_j(t)) = \delta^\ell_j.
\end{equation*}
Then, since the construction of the function space is tensor-product, the functionals 
\begin{equation*}
\begin{array}{ll}
	 & \left\{ \lambda^s_{i} \mu^{i}_{j}: 0\leq j\leq i \mbox{ and } 0\leq i\leq k \right\}  \\ 
	\cup & \left\{ \lambda^s_{i} \lambda^{t}_{j}: 0\leq j\leq N+p \mbox{ and } k+1\leq i\leq M+p \right\} 
\end{array}
\end{equation*}
form a dual basis for the basis $\mathbb{S}^k$ as given in Definition \ref{definitionSkhat}.

\section{Singular tensor-product patches of type B}\label{sectionTypeB}

In this section we discuss the construction of a smooth basis for singular parameterizations of type B, which have collinear parameter directions at a point of the boundary. A parameterization $\f G$ is called a \emph{singular mapping of type B} if the partial derivatives are collinear and in opposite direction at $(s,t) = (0,0)$, i.e. there exists a $\lambda>0$ such that 
\begin{equation*}
	\frac{\partial \f G}{\partial s} (0,0) = - \lambda \frac{\partial \f G}{\partial t} (0,0),
\end{equation*} 
leading to $\det \nabla \f G (0,0) = 0$. In the following we give a construction for B-spline function spaces leading to smooth isogeometric spaces, similar to the construction presented for patches of type A. 
\begin{definition}
Let $k\leq \min(p,q)$. The function space $\mathcal{S}^k(\f u,\f S) \subset \mathcal{S}$ is defined as the space of splines $f \in \mathcal{S}$, such that there exists a polynomial $\varrho \in \mathbb{P}_k$ fulfilling 
\begin{equation*}
	\frac{\partial^{|\alpha|} f }{\partial s^{\alpha_1}\partial t^{\alpha_2}} (\f s) = 
	\frac{\partial^{|\alpha|} (\varrho \circ \f u) }{\partial s^{\alpha_1}\partial t^{\alpha_2}} (\f s) \; \mbox{ for all } \f s \in \f S,
\end{equation*}
for $0 \leq \alpha_1, \alpha_2 \leq k$, $|\alpha| = \alpha_1+\alpha_2$, where $\f u$ is the mapping 
\begin{eqnarray*}
	\f u : \;\; [0,1]^2 & \; \rightarrow \; & \Delta = \{ (u,v): 0\leq u \leq 1 , \; -1+u\leq v\leq 1-u \} \\
	(s,t)^T & \; \mapsto \; & \left(s \, t, t - s \right)^T.
\end{eqnarray*}
and $\f S = \{(0,0)\}$.

Moreover, we can define a Bernstein-like basis $\mathbb{S}^k$ via 
\begin{equation*}
\begin{array}{ll}
	\mathbb{S}^k = & \left\{ \tilde{B}^k_{(i,j)}(s,t): 0\leq i,j\leq k \mbox{ and } i+j \leq k \right\} \\ 
	& \cup \left\{ B^p_i[\mathrm{S}](s)B^q_j[\mathrm{T}](t): 0\leq i\leq N_1+p, \; 0\leq j\leq N_2+q \mbox{ and } \max(i,j) > k \right\},
\end{array}
\end{equation*}
where $B^p_i[\mathrm{S}]B^q_j[\mathrm{T}]$, with $(0,0) \leq (i,j) \leq (N_1+p,N_2+q)$, is the standard basis of $\mathcal{S}$ and 
\begin{equation}
	\tilde{B}^k_{(i,j)} \in \mbox{span} \left(
	\left\{ B^p_{\ell_1}[\mathrm{S}](s)B^q_{\ell_2}[\mathrm{T}](t): 0\leq \ell_1,\ell_2 \leq k \right\} \right)
	\label{equationTildeBspan}
\end{equation}
are defined in such a way that 
\begin{equation*}
	\frac{\partial^{|\alpha|} \tilde{B}^k_{(i,j)} }{\partial s^{\alpha_1}\partial t^{\alpha_2}} (\f s) = 
	\frac{\partial^{|\alpha|} (\beta^k_{(i,j)} \circ \f u) }{\partial s^{\alpha_1}\partial t^{\alpha_2}} (\f s) \; \mbox{ for all } \f s \in \f S,
\end{equation*}
for $0 \leq \alpha_1, \alpha_2 \leq k$, $|\alpha| = \alpha_1+\alpha_2$, with triangular Bernstein basis functions 
\begin{equation*}
	\beta^k_{(i,j)} \; : \; \Delta \rightarrow \RR \; : \;  (u,v) \mapsto \frac{k!}{i! j! (k-i-j)!} \left(\frac{1-u-v}{2}\right)^i \left(\frac{1-u+v}{2}\right)^j u^{k-i-j}.
\end{equation*} 
\end{definition}
\begin{remark}
One can show easily, that $\mathbb{S}^k$ is in fact a basis for $\mathcal{S}^k(\f u,\f S)$. Moreover, similar to singular mappings of type A, the isogeometric function space $\mathcal{V}^k$ derived from $\mathcal{S}^k(\f u,\f S)$ fulfills $\mathcal{V}^k \subset \mathscr{C}^k(\overline{\Omega})$ if the underlying triangular patch is regular.
\end{remark}
We do not go into the details of the construction but present an example configuration allowing for $\mathcal{C}^2$-smooth bi-cubic isogeometric functions.
\begin{example}\label{exampleTypeB}
Let $p=q=3$ and $k=2$. We construct the geometry mapping from a bi-quadratic rational triangular patch representing a quarter of a circle. Applying degree elevation to a bi-quadratic triangular B\'ezier parameterization $\f T$ given by its homogeneous control points 
\begin{equation*}
	\begin{array}{lccc}
	\f t_{(i,j,2-i-j)} & j=0 & j=1 & j=2\\
	\hline 
	i=0 & (1,0,0)^T & (2,0,1)^T & (1,0,1)^T \\
	i=1 & (2,1,0)^T & (\sqrt{2},\sqrt{2},\sqrt{2})^T &  \\
	i=2 & (1,1,0)^T &  & 
	\end{array}
\end{equation*}
leads to a singular tensor-product patch $\f G = \f T \circ \f u$ with control points as depicted in Figure \ref{figureSingPatch}. The blue control points correspond to standard basis functions and the red control points correspond to basis functions $B^p_{\ell_1}[\mathrm{S}]B^q_{\ell_2}[\mathrm{T}]$, with $0 \leq \ell_1,\ell_2 \leq 2$, that span the space containing the new basis functions $\tilde{B}^k_{(i,j)}$, with $0 \leq i+j \leq 2$, as in equation (\ref{equationTildeBspan}). 
\begin{figure}[!ht]
    \centering
    \includegraphics[width=0.3\textwidth]{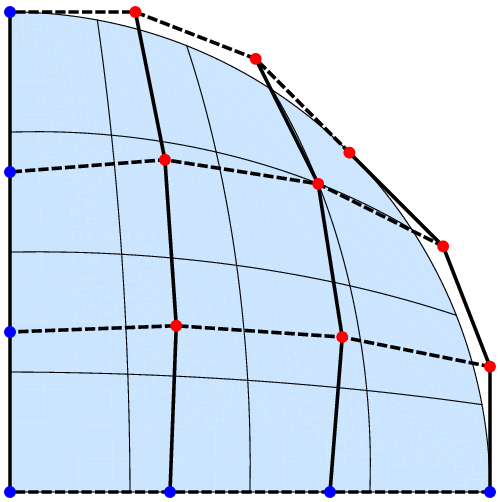}
    \caption{Parameterization and control points for a singular patch of type B for Example \ref{exampleTypeB}}\label{figureSingPatch}
\end{figure}
\end{example}
In the following we briefly discuss a way to generalize to higher dimensions.

\section{Constructions for higher dimension}\label{sectionHigherDim}

The approach presented here can also be generalized to higher dimensions. On the one hand one can generate smooth isogeometric function spaces on surfaces embedded in $\RR^3$ directly by substituting the planar triangular patch with a triangular surface patch. This may be for interest when dealing with partial differential equations on surfaces or for implementations of a boundary element method (e.g. \cite{Simpson2014}).

The basic idea behind this generalization is to consider $\f f = (g_0,g_1,g_2,g_3,f)^T$, with $g_i,f \in \mathcal{S}^k(\f u,\f S)$ for either type A or type B. Then, the isogeometric function 
\begin{eqnarray*}
	\varphi : \;\; \Omega & \; \rightarrow \; & \RR \\
	\f x & \; \mapsto \; & \frac{f}{g_0}\circ\left( \frac{g_1}{g_0},\frac{g_2}{g_0},\frac{g_3}{g_0} \right)^{-1}(\f x)
\end{eqnarray*}
defined on the surface $\Omega = \f G(\f B^\circ) \subset \RR^3$ is smooth of order $k$ if the underlying triangular surface patch is regular. 

On the other hand, one can define smooth isogeometric spaces on singularly parameterized volumetric domains. Similar to the bivariate case, one can again define an isogeometric function represented via its graph in homogeneous coordinates
\begin{equation*}
\f f = (g_0,g_1,g_2,g_3,f)^T : [0,1]^3 \rightarrow \tilde \Omega \times \RR
\end{equation*}
with $\tilde \Omega$ being the homogeneous representation of the physical domain $\Omega$. On a tetrahedral domain, given by the tri-linear singular mapping $\f u(r,s,t) = (r, r\,s , r\,s\,t)$, we can define a basis according to 
\begin{equation*}
	\frac{\partial^{|\alpha|} \tilde{B}^k_{\f i} }{\partial r^{\alpha_1}\partial s^{\alpha_2}\partial t^{\alpha_3}} (\f s) = 
	\frac{\partial^{|\alpha|} (\beta^k_{\f i} \circ \f u) }{\partial r^{\alpha_1}\partial s^{\alpha_2}\partial t^{\alpha_3}} (\f s) \; \mbox{ for all } \f s = (r,s,t) \in \f S,
\end{equation*}
for $\f S = \{0\}\times [0,1]\times [0,1]$, $0 \leq \alpha_1, \alpha_2, \alpha_3 \leq k$, $|\alpha| = \alpha_1+\alpha_2+\alpha_3$, and tri-variate tetrahedral Bernstein polynomials $\beta^k_{\f i}$ for $\f i = (i_1,i_2,i_3)$ with $0\leq i_1+i_2+i_3 \leq k$. Moreover, one needs to enforce smoothness along the face $s=0$ of the unit box, which collapses to a line in physical space. Such a construction corresponds to the findings in \cite{Takacs2014} about the smoothness conditions of isogeometric functions on volumetric patches.

\section{Conclusion}\label{sectionConclusion}

In this paper we presented a local mapping technique to construct isogeometric functions of arbitrary smoothness over singularly parameterized domains. The construction works for domains of arbitrary dimension. We focused on two dimensional patches containing exactly one point of singularity in physical space. However, the concept can be generalized to embedded surfaces and volumes as well as to structurally more complex domains. 

One direction of future research is the study and development of a refinement scheme maintaining the smoothness without enforcing additional smoothness conditions after each refinement step. Another area of interest, which arises for the presented isogeometric function spaces, is the question of approximation properties and convergence behavior. 
It is not clear, following the presented construction, whether or not the approximation properties of the function space are optimal. 

\section{Acknowledgments}

The work presented here is partially supported by the Italian MIUR through the FIRB ``Futuro in Ricerca'' Grant RBFR08CZ0S and by the European Research Council through the FP7 Ideas Consolidator Grant \emph{HIgeoM}. This support is gratefully acknowledged.


\begin{thebibliography}{10}

\bibitem{BeiraodaVeiga2011}
L.~Beir\mbox{\~{a}}o~da Veiga, A.~Buffa, J.~Rivas, and G.~Sangalli.
\newblock Some estimates for {h}-{p}-{k}-refinement in isogeometric analysis.
\newblock {\em Numerische Mathematik}, 118:271--305, 2011.

\bibitem{Benson2010}
D.~J. Benson, Y.~Bazilevs, M.~C. Hsu, and T.~J.~R. Hughes.
\newblock Isogeometric shell analysis: The {R}eissner-{M}indlin shell.
\newblock {\em Computer Methods in Applied Mechanics and Engineering},
  199(5-8):276 -- 289, 2010.

\bibitem{Cottrell2006}
J.~A. Cottrell, A.~Reali, Y.~Bazilevs, and T.~J.~R. Hughes.
\newblock Isogeometric analysis of structural vibrations.
\newblock {\em Computer Methods in Applied Mechanics and Engineering},
  195(41-43):5257 -- 5296, 2006.

\bibitem{Farin1999}
G.E. Farin.
\newblock {\em NURBS: from projective geometry to practical use}.
\newblock Ak Peters Series. A.K. Peters, 1999.

\bibitem{Hu2001}
S.-M. Hu.
\newblock Conversion between triangular and rectangular {B}\'ezier patches.
\newblock {\em Computer Aided Geometric Design}, 18(7):667 -- 671, 2001.

\bibitem{Hughes2005}
T.~J.~R. Hughes, J.~A. Cottrell, and Y.~Bazilevs.
\newblock Isogeometric analysis: {CAD}, finite elements, {NURBS}, exact
  geometry and mesh refinement.
\newblock {\em Computer Methods in Applied Mechanics and Engineering},
  194(39-41):4135 -- 4195, 2005.

\bibitem{Kiendl2010}
J.~Kiendl, Y.~Bazilevs, M.-C. Hsu, R.~W\mbox{\"{u}}chner, and K.-U. Bletzinger.
\newblock The bending strip method for isogeometric analysis of
  {K}irchhoff-{L}ove shell structures comprised of multiple patches.
\newblock {\em Computer Methods in Applied Mechanics and Engineering},
  199(37-40):2403 -- 2416, 2010.

\bibitem{Kiendl2009}
J.~Kiendl, K.-U. Bletzinger, J.~Linhard, and R.~W\mbox{\"u}chner.
\newblock Isogeometric shell analysis with {K}irchhoff-{L}ove elements.
\newblock {\em Computer Methods in Applied Mechanics and Engineering},
  198(49-52):3902 -- 3914, 2009.

\bibitem{Lu2009}
J.~Lu.
\newblock Circular element: Isogeometric elements of smooth boundary.
\newblock {\em Computer Methods in Applied Mechanics and Engineering},
  198(30-32):2391 -- 2402, 2009.

\bibitem{PieglTiller1995}
L.~Piegl and W.~Tiller.
\newblock {\em The {NURBS} book}.
\newblock Springer, London, 1995.

\bibitem{Prautzsch2002}
H.~Prautzsch, W.~Boehm, and M.~Paluszny.
\newblock {\em {B}\mbox{\'{e}}zier and {B}-Spline Techniques}.
\newblock Springer, New York, 2002.

\bibitem{Schumaker2007}
L.~L. Schumaker.
\newblock {\em Spline Functions: Basic Theory}.
\newblock Cambridge University Press, Cambridge, 2007.

\bibitem{Simpson2014}
R.N. Simpson, M.A. Scott, M.~Taus, D.C. Thomas, and H.~Lian.
\newblock Acoustic isogeometric boundary element analysis.
\newblock {\em Computer Methods in Applied Mechanics and Engineering},
  269(0):265 -- 290, 2014.

\bibitem{Takacs2011}
T.~Takacs and B.~J\mbox{\"u}ttler.
\newblock Existence of stiffness matrix integrals for singularly parameterized
  domains in isogeometric analysis.
\newblock {\em Computer Methods in Applied Mechanics and Engineering},
  200(49-52):3568--3582, 2011.

\bibitem{Takacs2012-1}
T.~Takacs and B.~J\mbox{\"u}ttler.
\newblock ${H}^2$ regularity properties of singular parameterizations in
  isogeometric analysis.
\newblock {\em Graphical Models}, 74(6):361--372, 2012.

\bibitem{Takacs2014}
Thomas Takacs, Bert J\mbox{\"u}ttler, and Otmar Scherzer.
\newblock Derivatives of isogeometric functions on $n$-dimensional rational
  patches in $\mathbb{R}^d$.
\newblock {\em Computer Aided Geometric Design}, 31(7-8):567--581, 2014.
\end{thebibliography}
\end{document}